\documentclass[11pt,a4paper]{article}

\usepackage[dvipsnames]{xcolor}
\usepackage[all,cmtip]{xy}
\entrymodifiers={+!!<0pt,\fontdimen22\textfont2>}

\usepackage{fouriernc}
\usepackage[english]{babel}         %outputs according to rules of language
\usepackage{verbatim}               %outputs text verbatim via \begin{verbatim}\end{verbatim}
\usepackage[T1]{fontenc}

\usepackage{amsmath}
\usepackage{amsfonts}
\usepackage{amssymb}
\usepackage{mathrsfs}    %allows \mathscr
\usepackage{amsthm}

\usepackage[colorlinks=true,linkcolor=MidnightBlue,citecolor=OliveGreen]{hyperref}           %Hyperlinks

\swapnumbers

\theoremstyle{plain}
\newtheorem*{thrm}{Theorem}
\newtheorem{thm}{Theorem}[section]

\newtheorem{prop}[thm]{Proposition}
\newtheorem{cor}[thm]{Corollary}
\newtheorem{lemma}[thm]{Lemma}

\theoremstyle{definition}

\newtheorem{defn}[thm]{Definition}
\newtheorem{defns}[thm]{Definitions}

\theoremstyle{remark}
\newtheorem{remark}[thm]{Aside}
\newtheorem{eg}[thm]{Example}

\newcommand{\proofof}[1]{\end{#1}\begin{proof}}

\makeatletter
\renewcommand\section{\@startsection {section}{1}{\z@}%
  {-3.5ex \@plus -1ex \@minus -.2ex}{2.3ex \@plus.2ex}%
  {\normalfont\large\bfseries}}
\renewcommand\subsection{\@startsection{subsection}{2}{\z@}%
  {-3.25ex\@plus -1ex \@minus -.2ex}{1.5ex \@plus .2ex}%
  {\normalfont\bfseries}}

\newcommand{\sh}[1]{\mathcal{#1}}
\newcommand{\N}{{\mathbb N}}
\newcommand{\Z}{{\mathbb Z}}
\newcommand{\Q}{{\mathbb Q}}
\newcommand{\R}{{\mathbb R}}

\newcommand{\F}{{\mathbb F}}
\newcommand{\B}{{\mathbb B}}
\newcommand{\G}{{\mathbb G}}
\renewcommand{\P}{{\mathbb P}}

\newcommand{\tens}{\mathbin{\otimes}}

\newcommand{\Hom}{\mathrm{Hom}}

\newcommand{\GL}{\mathrm{GL}}

\newcommand{\SL}{\mathrm{SL}}
\newcommand{\Aff}{\mathrm{Aff}}

\DeclareMathOperator*\colim{colim}
\DeclareMathAlphabet{\mathrmsl}{OT1}{cmr}{m}{sl}

\newcommand{\rssymb}[2]{\newcommand{#1}{\mathrmsl{#2}} }
\newcommand{\oper}[3][n]{\newcommand{#2}{\mathop{\mathrm{#3}}%
\ifx n#1\nolimits\else\limits\fi} }
\newcommand{\rsoper}[3][n]{\newcommand{#2}{\mathop{\mathrmsl{#3}}%
\ifx n#1\nolimits\else\limits\fi} }

\oper\Ad{Ad}
\oper\ad{ad}

\oper\val{val}
\oper\coker{coker}
\oper\mult{mult}
\oper\Iso{Iso}
\oper\End{End}
\oper\Aut{Aut}
\oper\Sub{Sub}
\oper\Alt{Alt}
\oper\Ext{Ext}
\oper\Pic {Pic}
\oper\Sym{Sym}
\oper\Spec{Spec}
\oper\Spf{Spf}
\oper\Sp{Sp}
\oper\Spa{Spa}
\oper\Proj{Proj}
\rsoper\divg{div}
\rsoper{\sym}{sym}
\rsoper{\alt}{alt}
\rsoper\trace{tr}
\rssymb\id{id}

\newcommand{\thismonth}{\ifcase\month\or
  January\or February\or March\or April\or May\or June\or
  July\or August\or September\or October\or November\or December\fi
  \space\number\year}

\usepackage[disable]{todonotes}

\newcommand{\CPA}{\mathrm{CPA}}

\newcommand{\Alg}{\mathrm{Alg}}

\newcommand{\Mod}{\mathrm{Mod}}
\newcommand{\Prim}{\mathrm{Prim}}
\newcommand{\Poly}{\mathbf{Poly}}
\newcommand{\V}{\mathbb V}

\title{Projective modules over polyhedral semirings}
\author{Andrew W. Macpherson}
\begin{document}

\maketitle
\begin{abstract}I classify projective modules over idempotent semirings that are free on a monoid. The analysis extends to the case of the semiring of convex, piecewise-affine functions on a polyhedron, for which projective modules correspond to convex families of weight polyhedra for the general linear group.
\end{abstract}
\setcounter{tocdepth}{1}
\tableofcontents

\section{Introduction}

This paper begins the geometric study of module theory over a class of idempotent semirings that are of basic importance in skeletal and tropical geometry: those generated freely by a monoid of \emph{monomials}. The projective modules over these semirings can be described in terms of a simpler category of partially ordered modules over the underlying monoid.

In the related case of the semiring of convex, piecewise-affine functions on a polyhedron, this latter category can itself be realised in terms of convex geometric data. This last intepretation of the classification suggests explicit descent criteria for modules over such `polyhedral' semirings.

The only existing work of which I am aware in this direction is \cite{crumpy}, which addresses finite projective modules over the real semifield $\R_\vee=\R\sqcup\{-\infty\}$ (therein denoted $\mathbb T$). As far as I know, the classification scheme given in the present paper is new even in that case.

\subsection*{Results}
The moral of the paper will be that over semirings generated freely by monomials, projective modules are defined by monomial inequalities. 

As a warm-up, we have the following result for the case of the simplest possible semiring, the Boolean semifield $\B=\{-\infty,0\}$:

\begin{thrm}[\ref{B_PRIM_THM}]Let $\mu$ be a finite $\B$-module. The following are equivalent:
\begin{enumerate}\item $\mu$ is projective;
\item $\mu$ is dualisable;
\item $\mu$ is flat;
\item $\mu$ is free on a poset;
\item $\mu$ has unique irredundant primitive decompositions.
\end{enumerate}\end{thrm}

%The only non-trivial part of this theorem is the equivalence \emph{i)}$\Leftrightarrow$\emph{iv)}. 

A careful extension of the same analysis helps us to understand projectivity of modules over the \emph{free} or \emph{fractional ideal} semiring $\B[A;A^+]$ associated to a pair of $\F_1$-algebras (monoids with zero) $(A;A^+)$. We carry out this analysis in \S\ref{MON}. In the case that $A$ is a \emph{domain}, one obtains a complete classification in terms of \emph{partially ordered $(A;A^+)$-modules}. I do not reproduce the classification here, but refer the reader to corollary \ref{MON_DOMAIN_FIRST}.

\

The geometric part \S\ref{POLY} of the paper concerns modules over the semifield $H_\vee=H\sqcup\{-\infty\}$ associated to a totally ordered group $H\subseteq\R$.

\begin{thrm}[\ref{POLY_WEIGHT_COR}]Let $H\subseteq\R$ be an additive subgroup. The category of finite projective $H_\vee$-modules is anti-equivalent to a category of extended $H$-integral general linear weight polyhedra and convex, piecewise-affine maps whose linear parts are fundamental weights.\footnote{This result is closely related to the more specific theorem 1.5 of \cite{crumpy}, which states that every projective $\R_\vee$-submodule is isomorphic to a submodule of $\R_\vee^n$ closed under co-ordinate-wise \emph{minimum}, as well as maximum. By \cite[Thm. B]{Xtracrumpy}, such sets are automatically convex polyhedra, and it is an elementary matter to observe what kinds of supporting half-spaces are allowed. Moreover, \cite{crumpy} even provides criteria to determine when a given submodule of $\R_\vee^n$ is projective.}\end{thrm}

By duality for finite projectives these categories are also \emph{equivalent} - I have just found it more natural to phrase the result in the form of a duality.

Taking $H=0$ in this result recovers a more geometric formulation of the (only non-trivial) equivalence \emph{i)}$\Leftrightarrow$\emph{iv)} of theorem \ref{B_PRIM_THM}: finite projective $\B$-modules are dual to certain convex cones in the Coxeter complex of $\GL_n$.

\

We cannot immediately extend our classification scheme to the more geometric setting of the semiring of convex, piecewise-affine functions on a polyhedron, since the latter is not actually free on the group of affine functions. Rather, it is the \emph{normalisation} of the free semiring \cite{norm}. Making this replacement helps us to get a geometric classification:

\begin{thrm}[\ref{POLY_FAM_COR}]Let $\Delta$ be an $H$-rational polytope. The category of finite projective $\CPA(\Delta,\Z)$-modules is anti-equivalent to a category of convex families of $\GL$ weight polyhedra over $\Delta$.\end{thrm}

For a more precise description of the latter category and the duality, see \S\ref{POLY_FAM}.

\subsection*{Acknowledgement}
I'd like to thank Jeff Giansiracusa, a conversation with whom gave me the idea for this paper.

\section{Preliminaries}\label{ALG}

\subsection{Points}\label{ALG_PTS}
It will be convenient to switch between the categories of pointed and unpointed sets. We do this using the strongly monoidal functor
\[ -\tens\F_1:(\mathbf{Set},\times)\rightarrow (\mathbf{Set}_*,\tens_{\F_1}=\wedge)\]
that adjoins a disjoint base point. This exhibits the category $\mathbf{Set}_*$ of pointed sets as the universal way to attach a zero object to $\mathbf{Set}$. 

This functor extends to the pointed versions of all essentially algebraic theories in $\mathbf{Set}$. In particular, a monoid $Q$ can be replaced by a monoid with zero $\F_1[Q]$, and a partially ordered set by a pointed partially ordered set (the point is the minimum). These constructions have the same universal property: $\mathbf{POSet}_*$ (resp. $\Mod_{\F_1[Q]}$) is the universal pointed extension of $\mathbf{POSet}$ (resp. $\Mod_Q$).

We will adhere to the convention of writing monoids without zero (mainly appearing in \S\ref{POLY}) and idempotent semirings \emph{additively}, and pure $\F_1$-algebras (\S\ref{MON}) multiplicatively. If $(Q;Q^+)$ is a pair of unpointed monoids, we will write $\F_1[Q]$ as a shorthand for the associated $\F_1$-algebra pair. When $H\subseteq\R$ is an additive subgroup, $\F_1[H]$ denotes the $\F_1$-algebra pair that in \cite{rig2} was (more suggestively) labelled $\F_1(\!(t^{-H})\!)$.

\subsection{Projectives}
Let $A$ be a commutative monoid (with or without zero) or semiring. (In fact, the following definition is standard for any commutative algebraic monad in the sense of \cite{Durov}.)
\begin{defn}An $A$-module is \emph{projective} if it satisfies the equivalent conditions
\begin{enumerate}
\item $\Hom(P,-)$ is exact;
\item any surjection $F\twoheadrightarrow P$ has a section;
\item $P$ is a retract of a free module.
\end{enumerate}\end{defn}

The situation when $A$ is a monoid - with or without zero - is very simple. A free $A$-module splits \emph{uniquely} as a sum of cyclic factors
\[ M\cong \bigoplus_i A_i \]
with each $A_i\simeq A$ (where the coproduct $\oplus$ is disjoint union in the unpointed case and wedge sum when $0\in A$). The \emph{factor} of a non-zero element $x\in M$ is the index of the cyclic submodule to which it belongs. This splitting, and in particular, its index set, is natural in $M$. In other words, every matrix over $A$ is a product of a permutation and a diagonal matrix. It follows:

\begin{prop}\label{ALG_CYCLIC}Let $A$ be a commutative monoid. Every projective $A$-module splits uniquely as a coproduct of cyclic submodules, each isomorphic to the image of an idempotent in $A$.\end{prop}

Let us call a monoid (resp. monoid with zero $A$) a \emph{domain} if it is cancellative (resp. $A\setminus 0$ is a cancellative submonoid).

\begin{cor}If $A$ is a domain, then every projective $A$-module is free.\end{cor}

\subsection{Flatness and duality over semirings}
We now restrict attention to the case that $A$ is a semiring, whence $\Mod_A$ is semiadditive (i.e.\ finite coproducts are products).

\begin{defns}Let $A$ be a (not necessarily idempotent) semiring, $M$ an $A$-module. $M$ is said to be:
\begin{enumerate}
\item \emph{dualisable} if there exists a module $M^\vee$ such that $-\tens_AM^\vee$ is adjoint to $-\tens_AM$;
\item \emph{flat} if $-\tens_AM$ is exact.
\end{enumerate}%In general, projective or dualisable implies flat, while finitely generated projective modules are dualisable. Indeed, finitely generated free modules are self-dual, and any retract of a dualisable object is dualisable.

The \emph{linear dual} of a module is the object
\[ M^\vee:=\Hom_A(M,A); \]
a module is said to be \emph{reflexive} if $M\tilde\rightarrow (M^\vee)^\vee$. This is strictly weaker than being dualisable, for while we always have the evaluation map $M\tens_A M^\vee\rightarrow A$, there may be no `identity matrix' $A\rightarrow M^\vee\tens_AM$.\end{defns}

\begin{remark}Flatness, at least with this definition, does not make much sense when $A$ is a monoid (more generally, when $\Mod_A$ is not semiadditive), because the class of flat modules is not closed under finite coproducts. In particular, free modules on more than one generator are never flat.
Some authors \cite{bougie} have studied variants of the notion of flatness adapted to $A$-modules (or `$A$-acts') in the unpointed case.\end{remark}

The relations between these properties are, much as in the case of commutative rings, as follows:
\begin{itemize}\item Any coproduct, filtered colimit, or retract of a flat module is flat. In particular, projective modules are flat.
\item Any finite coproduct or retract of a dualisable module is dualisable. In particular, finitely generated projective modules are dualisable.
\item Dualisable modules are flat, since in this case $-\tens_A M^\vee$ is left adjoint to $-\tens_AM$.
\item Conversely, finitely \emph{presented} flat modules are projective (corollary \ref{ALG_FLAT}).\end{itemize}

The following fact is no doubt well-known - indeed, the proof for the case of rings \cite[\href{http://stacks.math.columbia.edu/tag/00HK}{00HK}]{stacksproject} carries through with only minor modifications.

\begin{prop}[Equational criterion for flatness]\label{ALG_EQN}Let $M$ be a flat module, $v:A^n\rightarrow M$ a homomorphism from a finite free module. Suppose that $v$ satisfies a relation $f\sim g$:
\[ A \underset{g}{\stackrel{f}{\rightrightarrows}} A^n \stackrel{v}{\rightarrow} M. \]
Then $v$ factors through a finite free module in which $f\sim g$.\end{prop}
\begin{proof}%The objective is to show there exists a matrix $\Phi:A^n\rightarrow A^m$ satisfying the relation $f\sim g$ and factoring $A^n\rightarrow M$.
Dualising the relation, we obtain an equaliser sequence
\[ K\rightarrow A^n\rightrightarrows A, \]
and hence by flatness of $M$, an equaliser
\[ K\tens M \rightarrow M^n\rightrightarrows M. \]
It follows that $v\in M^n$ actually lies in $K\tens M$.

By writing $v$ as a sum of decomposable elements, one obtains morphisms $A^m\rightarrow K\subseteq A^n$ and $A^m\rightarrow M$ that induce the composite
\[ A^m \stackrel{\Sigma}{\rightarrow} A \stackrel{v}{\rightarrow} K\tens M \]
on the tensor product. The transpose of $A^m\rightarrow A^n$ is the desired factorisation.
\end{proof}
\begin{cor}Every homomorphism from a finitely presented module into a flat module factors through a finite free module.\end{cor}
\begin{cor}\label{ALG_FLAT}Every finitely presented flat module is projective.\end{cor}
\begin{cor}[Lazard's theorem]\label{ALG_LAZARD}A module is flat if and only if it is a filtered colimit of finite free modules.\end{cor}

\section{$\B$-modules}\label{B}
For the rest of the document, all semirings will be additively idempotent. I continue with the convention of \cite{mac} in denoting idempotent semirings and their modules by lowercase Greek letters, and their operations by $\vee$ (`max') and $+$. Correspondingly, the closed monoidal structure on the category $\Mod_\alpha$ of modules over an idempotent semiring $\alpha$ is denoted $\oplus_\alpha$ (i.e.\ this does \emph{not} denote the coproduct of modules).

As in the introduction, we will denote by $\B$ the \emph{Boolean semifield}, the initial object in the category of idempotent semirings. It will be illuminating to understand $\B$ as a monad via its free functor
\[ \B:\mathbf{Set}\rightarrow\mathbf{Set} \]
that takes a set $S$ to the set $\B S$ of its finite subsets, with operations
\[ 1_\mathbf{Set} \stackrel{\{-\}}{\longrightarrow}\B  \stackrel{\bigcup}{\longleftarrow}\B\B \]
given by singleton and union, respectively. This monad is the unique algebraic (i.e. commuting with filtered colimits) extension of the monad of power set and union on the category of finite sets. A finite, free $\B$-module is nothing more than the power set of a finite set.

The same statements remain valid, \emph{mutatis mutandi}, with $\mathbf{Set}$ replaced by the category $\mathbf{Set}_*$ of pointed sets. The free $\B$-module functor factorises
\[ \mathbf{Set}\rightarrow\mathbf{Set}_*\rightarrow\Mod_\B.\]
We will use this in \S\ref{POLY} to apply the results of \S\ref{MON}, couched in the setting of monoids with zero, in the unpointed regime.

More concretely, a $\B$-module $(\mu,\vee)$ is nothing more than a \emph{join semilattice}, that is, a partially ordered set with finite joins, and a $\B$-linear morphism is a right exact monotone map. In \cite{mac}, these were called `spans'.

%\begin{lemma}\label{B_FINITE}
The following are equivalent for a $\B$-module $\mu$:
\begin{enumerate}\item $\mu$ is finite as an object of $\Mod_\B$;
\item $\mu$ is compact as an object of $\Mod_\B$;
\item $\mu$ is a finite set.\end{enumerate}

%(A subset $S_0$ of a poset $S$ is said to be \emph{lower} if $b\in S_0$ and $a\leq b$ implies $a\in S_0$.)

\subsection{Free module on a poset}
Let us denote by $\mathbf{POSet}$ the category of partially ordered sets and monotone maps. Since every $\B$-module is in particular a poset, we have a faithful, conservative functor
\[ \Mod_\B\rightarrow\mathbf{POSet}. \]
In fact, this functor is \emph{monadic}; its left adjoint takes a poset $(S,\leq)$ to the poset $\B(S,\leq)$ of finitely generated lower subsets. Since any union of lower subsets is lower, union makes this poset a $\B$-submodule of the power set $\sh P(S)$. It is called the \emph{free $\B$-module} on $(S,\leq)$.
%This `free' functor commutes with filtered colimits - that is, it forms an algebraic monad.

By general principles, $\B:\mathbf{POSet}\rightarrow\Mod_\B$ is strongly monoidal. It also respects Hom sets in the following way: if $S_1,S_2$ are posets, then the assignment
\[ (X,Y):Z\mapsto \left\{\begin{array}{ll} Y & \text{if }X\leq Z \\ -\infty & \text{otherwise}\end{array} \right. \]
defines a monotone map $S_1^\mathrm{op}\times S_2 \rightarrow \Hom(\B S_1,\B S_2)$; if $S_1$ is finite, then this map extends to an isomorphism 
\[ \B(S_1^\mathrm{op}\times S_2) \cong \Hom_\B(\B S_1,\B S_2) \]
so that maps $\B S_1\rightarrow\B S_2$ are finite monotone correspondences from $S_1$ to $S_2$.

The same logic holds over the category $\mathbf{POSet}_*$ of \emph{pointed} posets, that is, posets equipped with a distinguished minimum and monotone maps that preserve this minimum.

\paragraph{Flatness and projectivity}
The free module on a poset comes equipped with a natural set of generators
\[ \B S\twoheadrightarrow \B(S,\leq). \]
To understand projectivity, we will need to know when this map admits a section. A natural candidate for a splitting is the (right ind-adjoint) inclusion $\B(S,\leq)\hookrightarrow\sh P(S)$. As remarked above, it is automatically a $\B$-module homomorphism. This map factors through $\B S$ - thus defining an honest adjoint - if and only if $(S,\leq)$ satisfies the condition
\begin{itemize}\item any principal lower set $S_{\leq X}$ is finite.\hfill\emph{lower finite}\end{itemize} 
This has the flavour of a `finite presentation' condition for posets: it is equivalent that the quotient $\B S\rightarrow\B (S,\leq)$ be defined by finitely many relations. It is satisfied, in particular, whenever $S$ is finite.

It turns out that this is the only way to split this epimorphism:

%\begin{lemma}\label{B_POSET_SECTION}Let $p:(S_1,\leq)\rightarrow(S_2,\leq)$ be a monotone map, $\sigma:\B(S_2,\leq)\rightarrow\B(S_1,\leq)$ a monotone section. Then for each $X\in S_2$, $\sigma X\subseteq (S_1)_{\leq X}$ and $\sigma X\cap p^{-1}X\neq\emptyset$.In particular, if $p$ is bijective, then $\sigma$ is automatically a right adjoint and $\B$-linear.\end{lemma}
\begin{lemma}[Monotone sections]\label{B_POSET_SECTION}Let $p:(S_1,\leq)\rightarrow(S_2,\leq)$ be a bijective monotone map, $\sigma:\B(S_2,\leq)\rightarrow\B(S_1,\leq)$ a monotone section of its $\B$-linear extension. Then $\sigma$ is right adjoint to $\B p$ and $\B$-linear.\end{lemma}
\begin{proof}Let $\sigma:\B(S_2,\leq)\rightarrow \B(S_1,\leq)$ be any section, $S_0\in\B(S_2,\leq)$. Then $p\sigma S_0\subseteq S_0$ generates $S_0$ as a lower set. In particular, for any $X\in S_2$, $X\in p\sigma X$, i.e. $p^{-1}X\in\sigma X$.

If $\sigma$ is monotone, then this shows that $\sigma p$ is increasing on $\B(S_1,\leq)$. Thus $Y\leq\sigma X$ if and only if $pY\leq X$, that is, $\sigma$ is right adjoint to $p$.\end{proof}

The proof of the lemma \ref{B_POSET_SECTION} on monotone sections depends on the fact that a lower set has a unique minimal set of generators. We will use variations of the latter fact, and the lemma, repeatedly in the sequel.

\begin{prop}\label{B_POSET_PROJ}Let $(S,\leq)$ be a poset. The following are equivalent:
\begin{enumerate}\item $\B(S,\leq)$ is projective;
\item $\B(S,\leq)$ is lower finite;
\item $(S,\leq)$ is lower finite.\end{enumerate}\end{prop}

From Lazard's theorem \ref{ALG_LAZARD} it also follows:

\begin{cor}\label{B_PRIM_FLAT}The free $\B$-module on any poset is flat.\end{cor}

\begin{eg}[A flat module with no primitives]The converse to corollary \ref{B_PRIM_FLAT} is false: while every flat module is a filtered colimit of modules free on a poset, there is no requirement that the transition maps preserve these posets. For example, the set $\mu$ of compact open subsets of $\Z_2$ has an expression as a colimit \[\mu=\colim_k\B(\Z/2^k\Z) \simeq \colim_k\B^{2^k}\] with transition maps given by inverse image along $\Z/2^k\Z\twoheadrightarrow\Z/2^{k-1}\Z$, and is therefore flat by Lazard's theorem \ref{ALG_LAZARD}; however, it has no primitive elements (cf. \S\ref{B_PRIM} below) and so cannot be free on a poset.\end{eg}

\subsection{Primitives}\label{B_PRIM}
To achieve our goal of classifying projective modules, we must still characterise which $\B$-modules appear through this construction.

\begin{lemma}\label{B_PRIM_LEMMA}Let $X$ be an element of a $\B$-module $\mu$. The following are equivalent:
\begin{enumerate}\item $X=\bigvee_{i\in I}X_i$, with $I$ a finite set, implies $X=X_i$ for some $i\in I$;
\item if $S\subseteq\mu $ generates - that is, if $\B S\rightarrow\mu$ is surjective - then $X\in S$.\end{enumerate}\end{lemma}

\begin{defns}\label{B_PRIM_DEF}An element $X$ of a $\B$-module $\mu$ satisfying the equivalent conditions of lemma \ref{B_PRIM_LEMMA} is said to be $\vee$-\emph{primitive}, or simply \emph{primitive} if no confusion can arise. (Note that $-\infty$ is never primitive.) The set of primitive elements of $\mu$ is denoted $\mathrm{Prim}\mu$. It is usually not functorial in either direction.

We have tautological maps
\[ \B(\Prim\mu)\twoheadrightarrow\B(\Prim\mu,\leq)\rightarrow\mu \]
where $\Prim\mu\subset \mu$ carries the induced partial order. If these modules surject onto $\mu$, we say that $\mu$ \emph{has primitive decompositions}. 

A \emph{primitive decomposition} of an element $X\in\mu$ is a lift to $\B(\Prim\mu)$. Such a decomposition is said to be \emph{irredundant} if it is minimal in its fibre of $\B(\Prim\mu)\twoheadrightarrow\B(\Prim\mu,\leq)$. The set of possible primitive decompositions (resp. irredundant decompositions) of elements of $\mu$ is precisely $\B(\Prim\mu)$ (resp. $\B(\Prim\mu,\leq)$).

In equations, a primitive decomposition of $X\in\mu$ is an expression $X=\bigvee_{i\in I}X_i$ with $X_i$ primitive, and it is irredundant if there are no order relations among different $X_i$. \end{defns}

\begin{remark}Beware that an irredundant decomposition in $\mu$ is not necessarily \emph{minimal}: one may perfectly well have a relation
\[ X_1\vee X_2\vee X_3=X_1\vee X_2 \]
between primitives in $\mu$, but unless $X_3\leq X_2$ or $X_3\leq X_1$, both decompositions will be irredundant. Of course, in light of theorem \ref{B_PRIM_THM}, this cannot happen for \emph{projective} modules.\end{remark}

\paragraph{Frees}
A module is free if and only if $\B(\Prim\mu)\rightarrow\mu$ is an isomorphism; that is, if it has \emph{unique} primitive decompositions.

\begin{thm}\label{B_PRIM_FREE}Let $\mu$ be a $\B$-module. The following are equivalent:
\begin{enumerate}\item $\mu$ is free;
\item $\mu$ has unique primitive decompositions.
\end{enumerate}\end{thm}

\paragraph{Projectives}More generally, a module is free on a poset if and only if it has unique irredundant primitive decompositions.

\begin{thm}\label{B_PRIM_THM}Let $\mu$ be a $\B$-module. The following are equivalent:
\begin{enumerate}\item $\mu$ is projective;
\item $\mu$ is free on a lower finite poset;
\item $\mu$ has unique irredundant primitive decompositions, and the set of primitives is lower finite.
\end{enumerate}\end{thm}
\begin{proof}
To complete the proof of this theorem, we must show that projective modules have unique irredundant primitive decompositions.

\begin{lemma}\label{B_PRIM_SUB}Any submodule of a $\B$-module having primitive decompositions itself has primitive decompositions.\end{lemma}

\begin{lemma}[Sections over primitives]\label{B_PRIM_SECTION}Let $S\subseteq\mu$ be a subset, $\sigma$ any section of \[\B(S,\leq)\rightarrow\mu. \] If $X\in\mu$ is primitive, then $\sigma X=S_{\leq X}$.\end{lemma}
\begin{proof}Let $\sigma$ be a section, and let $X\in\mu$. Since $\bigvee_{Y\in\sigma X}Y=X$, $\sigma X$ must certainly be contained in $S_{\leq X}$. If $X$ is primitive, then in fact $X\in\sigma X$ and so $\sigma X=S_{\leq X}$.\end{proof}

A projective module $\mu$ is a submodule of a free module, and hence by lemma \ref{B_PRIM_SUB}, $\B(\Prim\mu,\leq)\rightarrow\mu$ is surjective. By projectivity, it admits a $\B$-linear section. Applying the lemma \ref{B_PRIM_SECTION} on sections over primitives to $S=\Prim\mu$ shows that it is an isomorphism.
\end{proof}

\section{Modules over free semirings}\label{MON}
Here we generalise and discuss modules over semirings that are free on a pair $(A;A^+)$ of monoids with $A$ an $A^+$-algebra. Our convention in this section will be that monoids are multiplicative with zero; the category of such pairs is denoted $\mathrm{Pair}_{\F_1}$. Following the remarks of \S\ref{ALG_PTS}, the results translate straightforwardly into the unpointed regime.

%\begin{defn}A homomorphism $M_1\rightarrow M_2$ between $A$-modules equipped with $A^+$-structures $M_i^+$ is \emph{bounded} if it is in the image of the embedding  \[ \Hom_{A^+}(M_1^+,M_2^+)\tens_{A^+}A\rightarrow\Hom_A(M_1,M_2). \]   If $A$ is a localisation of $A^+$ and $M_1$ is finitely generated, this map is an isomorphism and so boundedness is automatic. In infinite type, requiring boundedness of projectors affects the classification of projectives. For many applications, the correct category of modules over a pair has the set of bounded homomorphisms as maps. This will hardly affect the conclusions of this section.\end{defn}

Usually, $A$ will be a localisation of $A^+$; correspondingly, we will typically consider idempotent semirings $\alpha$ that are a localisation of their semiring of integers $\alpha^\circ:=\alpha_{\leq 0}$. However, these assumptions are not actually necessary for the conclusions of \S\S\ref{MON_FREE}-\ref{MON_PRIM}.

We will later need to assume that $A^+$ is integrally closed in $A$ (\S\ref{MON_DOMAIN}) and that $\alpha$ is \emph{normal} in the sense of \cite{norm} (\S\ref{POLY_FAM}).

%In \emph{loc.\ cit}.\ we also worked with \emph{topological} $\F_1$-algebras, but to simplify matters we will abdicate any such considerations in this paper.}
\

The structure of this section is as follows. The first two subsections and \S\ref{MON_PRIM} are parallel to the structure of \S\ref{B}. In \S\ref{MON_POMOD} we derive formal criteria for the free module on a partially ordered projective $A$-module to be projective, and in \S\ref{MON_PRIM} we show that conversely, every projective module over the free semiring $\B[A]$ on $A$ comes from a partially ordered $A$-module.

The technical heart of the paper lies in \S\ref{MON_QUIVER}, where we discuss a method of presenting partial orders via quivers. This will help us to unpack the meanings of the conditions appearing in \S\ref{MON_POMOD}. In particular, we define the notion of \emph{non-degeneracy} of partially ordered free modules, which has crucial finiteness implications.

When $A$ is a domain, we obtain a complete classification in \S\ref{MON_DOMAIN}. The remaining poings are, first, that the partially ordered $A$-module behind a projective $\B[A]$-module is itself projective, and second, that non-degeneracy is a necessary condition for lower finiteness.

\subsection{Free semirings and modules}\label{MON_FREE}
Let $(A;A^+)$ be an $\F_1$-algebra pair, and define, as in \cite{mac}, the \emph{fractional ideal semiring}
\[ \B[A;A^+]:=\{\text{finite }A^+\text{-submodules of }A\} \]
with semiring of integers
\[ \B[A;A^+]^\circ=\B[A^+;A^+].\]
We also write simply $\B[A]$ when $A^+$ is considered implicit. There is natural monoid homomorphism
\[ \log:A\rightarrow\B[A;A^+], \quad \log(A^+)\leq 0 \]
the `universal valuation' of $(A;A^+)$. This valuation is injective if and only if $A^+$ is sharp (i.e.\ has no invertible elements other than $1$). In general, the image of $A^+$ in the monoid semiring is its universal sharp quotient. For the purposes of studying $\B[A]$-modules, then, we can and will always assume that $A^+$ is sharp.

Note that in contrast to the situation for commutative rings, the fractional ideal functor has a monadic right adjoint \emph{forgetful functor}
\[ \Alg_\B \rightarrow \mathrm{Pair}_{\F_1},\quad \alpha \mapsto (\alpha;\alpha^\circ) \]
which forgets the $\vee$ operation. In this language, $\log$ is the unit of the adjunction. 

More generally, if $M$ is an $A$-module, one can form the \emph{disc} (or \emph{fractional submodule}) set
\[  \B(M;A^+):=\{\text{finite }A^+\text{-submodules of }M\}, \]
more briefly, $\B(M)$, which carries a natural structure of a $\B[A]$-module. This fits into an adjunction
\[ \B(-;A^+):\Mod_A\leftrightarrows\Mod_{\B[A]} \]
with the evident forgetful functor. %An $A^+$-structure for $M$ provides a $\B[A^+]=\B[A;A^+]^\circ$-structure for $\B(M)$, and vice versa when $A$ is a localisation of $A^+$. The extension of a bounded homomorphism is bounded. Thus the adjunction exists in the bounded regime as well.

It follows that $\B(-;A^+)$ preserves colimits, and hence the classes of free, projective, and (by the Lazard theorem \ref{ALG_LAZARD}) flat modules.

\subsection{Partially ordered modules}\label{MON_POMOD}
%In the context of monoids with zero, it will be convenient to work with \emph{pointed} posets, that is, posets equipped with a specified minimum $0$ (or $-\infty$ in additive notation) which monotone maps must respect. The category of such is denoted $\mathbf{POSet}_*$. The smash product $\tens_{\F_1}$ of two pointed posets is again a pointed poset, and makes $\mathbf{POSet}_*$ into a closed monoidal category.
If $A\in\mathrm{Pair}_{\F_1}$, then any $A$-module carries a natural pointed \emph{$A^+$-divisibility} partial order
\[ fv\leq_{A^+}v \quad \forall f\in A^+ \]
which is non-degenerate if and only if $A^+$ is sharp. An $A$-linear map is automatically monotone with respect to this order.

The divisibility order on a smash product is the smash product of the orders on the factors. It therefore makes $(A,\leq_{A^+})$ into an algebra object in the closed monoidal category $(\mathbf{POSet}_*,\tens_{\F_1})$ of pointed posets (cf. \S\ref{ALG_PTS}), and $(M,\leq_{A^+})$ into an $(A,\leq_{A^+})$-module. This defines fully faithful functors
\[ \mathrm{Pairs}_{\F_1}\hookrightarrow\mathrm{Alg}(\mathbf{POSet}_*,\tens_{\F_1})  \]
and, for each $(A;A^+)$,
\[ \Mod_A\hookrightarrow\Mod_{(A,\leq_{A^+})}(\mathbf{POSet}_*). \]

In the spirit of \cite{G2}, the free $\B[A]$-module on an $A$-module $M$ can be presented via this partial order
\[ \B(M;A^+)\cong\B(M;\leq_{A^+})=\B(M;\mathbf{Set}_*)/(fv\leq v|\forall f\in A^+) \]
which shows that, in particular, $\B(M)$ has unique irredundant primitive decompositions. By corollary \ref{B_PRIM_FLAT}, it is moreover flat as a $\B$-module.

\begin{defn}A \emph{partially ordered module}, or \emph{po-module}, over a $\F_1$-algebra pair $(A;A^+)$ is an $(A,\leq_{A^+})$-module object in $(\mathbf{POSet}_*,\tens_{\F_1})$. That is, it is an $A$-module together with a partial order such that
\begin{enumerate}\item  $fv\leq gv$ for all $v\in M$, $g\in A$ and $f\in A^+g$;
\item $v\leq w\Rightarrow fv\leq fw$ for all $v,w\in M$ and $f\in A$.\end{enumerate} 
The category of partially ordered $A$-modules is abbreviated  $\mathbf{PO}\Mod_A:=\Mod_{(A,\leq_{A^+})}(\mathbf{POSet}_*)$. It is closed monoidal and compactly generated.\end{defn}

By analogy with the case of $\B$-modules, we can define a `free' $\B[A]$-module on any partially ordered $A$-module:
\[ \B(M,\leq):=\{\text{finitely generated lower $A^+$-submodules of $M$}\}; \]
here a lower submodule is `finitely generated' if it is the lower hull of a finite $A^+$-submodule; and hence a monadic adjunction
\[ \mathbf{PO}\Mod_A\leftrightarrows\Mod_{\B[A]} \]
extending the one defined in \S\ref{MON_FREE}. Intuitively, $\B[A]$-modules that are free over a po-$A$-module are those `defined by monomial relations'.

\begin{prop}\label{MON_POMOD_THM}The free module on a partially ordered projective $A$-module $(M,\leq)$ is a projective $\B[A]$-module if and only if the following conditions are satisfied:
\begin{enumerate}\item the lower hull of a finite disc in $M$ is finite; \hfill \emph{lower finite}
\item if $D\subseteq M$ is a lower disc, then so is $fD$ for $f\in A$. \hfill \emph{lower saturated}
%\item the $A^+$-structure $M^+$ is commensurable with its lower hull. \hfill\emph{lower bounded}
\end{enumerate}\end{prop}
\begin{proof}Indeed, we have already observed that $\B(-;A^+)$ preserves projectivity, and so assuming $M$ projective, by the monotone section lemma \ref{B_POSET_SECTION} we must check that the canonical generator
\[ \B(M;A^+)\twoheadrightarrow \B(M,\leq) \]
has a right adjoint. Condition \emph{i)} is enough to show that the right adjoint exists as a map of posets, and \emph{ii)} is the condition that it be $A$-linear. %Condition \emph{iii)} implies that the section is bounded.
\end{proof}

%If we are working with finitely generated modules, or allowing unbounded homomorphisms, then of course we can ignore item \emph{iii)}.
We will expand upon the meanings of the other two conditions in \S\ref{MON_QUIVER}.

\begin{cor}A $\B[A]$-module is flat if it is free on a lower cancellative po-$A$-module that is flat as an $A$-module.\end{cor}

\subsection{Partial orders from quivers}\label{MON_QUIVER}In the case of projective $A$-modules - or more generally, direct sums of cyclic modules - we can give a fairly explicit method for defining partial orders. We will use this method as an auxiliary tool to obtain a good classification theorem \ref{MON_DOMAIN_FIRST}.

Let $M\cong \bigoplus_{i\in I} A_i$ be such a module, and let $\sh Q$ be a (possibly infinite) quiver with vertex set $I$. Suppose that we are given the structure of a \emph{representation} of $\sh Q$ on $M$ - that is, for each edge in $\sh Q$ from $i$ to $j$, a map $A_i\rightarrow A_j$. If we choose generators $x_i\in A_i$ for the cyclic factors of $M$ - since $A^+$ is sharp, this is the same as choosing a cyclic $A^+$-structure for each factor - then such can be represented by attaching an element of $A\twoheadrightarrow\Hom_A(A_i,A_j)$ to each edge of $\sh Q$.

With a choice of generators, we can exchange a quiver representation for a presentation
\begin{align}\label{q1} \{(ex_i,x_i)|e\in\mathrm{Edge}(\sh Q)\}\subset M^2 \end{align}
of a partial order on $M$. In particular, this presentation is finite if and only if $\sh Q$ has finitely many edges. Conversely, any presentation $R\subset M^2$ of this form - with the right-hand term always the chosen generator of its cyclic factor - can be obtained from a quiver representation, for which $R$ is the set of edges via the natural projection $M\setminus0\rightarrow I$.

More invariantly, the action of the path category $\mathrm{Path}(\sh Q)$ of $\sh Q$ defines an $A$-invariant pre-order  
\begin{align} \label{q2}x\leq^{\sh Q}\!\! y \quad\Leftrightarrow\quad \exists\gamma\in\mathrm{Path}(\sh Q):x=\gamma y \end{align} on $M$. It is the transitive and $A$-invariant closure of the relation defined by the set in (\ref{q1}).

The union of this pre-order with the $A^+$-divisibility order is an $A$-module pre-order
\begin{align}\label{q3} x\leq^\sh{Q}_{A^+}\!\! y \quad\Leftrightarrow\quad \exists\gamma\in\mathrm{Path}_{A^+}(\sh Q):x=\gamma y \end{align}
where $\mathrm{Path}_{A^+}$ denotes the $A^+$-linear extension of the path category. The latter is degenerate if there is a cycle $\gamma:i\rightarrow i$ in $\mathrm{Path}(\sh Q)$ such that $1\leq_{A^+}\gamma\in\End(A_i)$; otherwise, it defines the structure of a po-$A$-module on $M$. We call it the partial order \emph{presented by $\sh Q\curvearrowright M$}.

\paragraph{Lower saturation}Partial orders on $M$ can be equivalently described by producing, for each $x\in M$, the \emph{lower hull} $M_{\leq x}$. For \emph{lower saturated} partial orders (cf. proposition \ref{MON_POMOD_THM}, \emph{ii}) it is enough to define the lower hull for $x_i$ the generators of the cyclic factors of $M$, since by definition in that case,
\[ M_{\leq fx_i}=fM_{\leq x_i} \]
gives the formula for general elements $fx_i$ of $M$.

From such a module we can produce a quiver whose edges $i\rightarrow j$ are a generating set for $M_{\leq x_i}\cap A_j$ as an $A^+$-module. This quiver presents the partial order via (\ref{q1}).

Conversely, the partial order presented by a quiver action $\sh Q\curvearrowright M$ has lower sets
\begin{align}\label{lower} M_{\leq x}=\mathrm{Path}_{A^+}(\sh Q)x = A^+\mathrm{Path}(\sh Q)x \end{align}
whose definition manifestly commutes with the action of $A$. In other words, quiver partial orders are lower saturated.

\begin{prop}[Lower saturation]\label{MON_DOMAIN_SECOND}Let $(M,\leq)$ be a partially ordered direct sum of cyclic $A$-modules. The following are equivalent:
\begin{enumerate}\item $(M,\leq)$ is lower saturated;
\item $\leq$ can be presented by a quiver.
%\item ($A$ a domain) $M$ is a lower submodule of $M\tens_AK_A$, in particular carrying the induced order.
\end{enumerate}\end{prop}

\paragraph{Lower finiteness}When $(M,\leq)$ is lower finite, in particular each $M_{\leq x_i}$ is finitely generated, and so by sharpness of $A^+$ has a unique set of primitive generators. Let us call the quiver $\sh Q$ with these generators as its set of edges the \emph{canonical quiver}. (If $A^+_{\leq i}\cap A_i^+=A_i^+$, then this algorithm yields a \emph{trivial loop} at $i$, which we may exclude.) It has finitely edges departing from each vertex. In particular:

\begin{lemma}Let $M$ be finite and $\leq$ lower finite. Then $(M,\leq)$ is finitely presented.\end{lemma}

%The path algebra of \emph{any} presenting quiver must contain the canonical quiver. We will use a version of this observation in the proof of theorem \ref{MON_DOMAIN_THIRD}.

%\begin{lemma}\label{MON_DOMAIN_COMPACT}A partially ordered, finite, free $A$-module is:
%\begin{enumerate}\item compact as an object of $\mathbf{PO}\Mod_A$ if and only if it can be presented by a quiver;
%\item lower finite if and only if the canonical graph is a quiver.\end{enumerate}\end{lemma}

\paragraph{Non-degeneracy}Unfortunately, being presented by a quiver with finitely many edges departing each vertex is not sufficient to guarantee lower finiteness; by the formula \ref{lower}, it is the action of the path algebra at $i$ that we need to worry about, and the latter is infinite whenever there is a cycle at $i$. We need a way to disregard such cycles.

\begin{defn}[Non-degeneracy]Let $M$ be a direct sum of cyclic $A$-modules. A partial ordering on $M$ is said to be \emph{non-degenerate} if the induced partial order on each cyclic factor is the $A^+$-divisibility order $\leq_{A^+}$.\end{defn}

Non-degeneracy of a partial order entails that the element of $A$ attached to any path in a presenting quiver $\sh Q$ with the same start and end point $i$ actually lies in $A^+\subseteq \End(A_i)$, and so
\[ M_{\leq x_i}\cap A_i=A^+x_i \]
for any $x_i\in A_i$. Such cycles can therefore be disregarded in the presentation (\ref{q3}), and we may restrict attention to the set $\sh Q^\emptyset\subseteq\mathrm{Path}(\sh Q)$ of \emph{acyclic paths}. The latter is finite as soon as $\sh Q$ is.

\begin{lemma}\label{MON_DOMAIN_NONDEGEN}Let $(M,\leq)$ be lower saturated and finitely presented. If $\leq$ is non-degenerate, then it is lower finite.\end{lemma}

All that remains to obtain an equivalence 
\[\text{lower finite} \qquad \Leftrightarrow \qquad \text{finitely presented } + \text{ non-degenerate} \]
is to show that lower finite orders are non-degenerate. We conclude this for \emph{domains} in \S\ref{MON_DOMAIN}.

\subsection{Primitives}\label{MON_PRIM}
Free modules on a po-module have the following distinguishing features:
\begin{itemize}\item they have unique irredundant primitive decompositions (cf. \S\ref{B_PRIM});
\item the set of primitive elements, together with $-\infty$, is an $A$-submodule.\end{itemize}
If, in general, these criteria are satisfied, we can by passing to primitive elements reverse the procedure and obtain a po-$A$-module from a $\B[A]$-module. This is possible for finite projective $\B[A]$-modules, by a partial converse to proposition \ref{MON_POMOD_THM}:

\begin{lemma}\label{MON_PRIM_THM}Every projective $\B[A]$-module is free on a partially ordered $A$-module.
\end{lemma}
\begin{proof}We have seen that free $\B[A]$-modules are, in particular, free $\B$-modules on a poset. By this and lemma \ref{B_PRIM_SUB}, any projective $\B[A]$-module has primitive decompositions.

Now let $\mu$ be projective, $A\Prim\mu\subseteq \mu$ the $A$-submodule generated by the primitive elements.  Then $A\Prim\mu$ generates $\mu$ as a $\B$-module, that is,
\[ \B(A\Prim\mu,\leq)\twoheadrightarrow\mu\]
is surjective. By projectivity of $\mu$, it has a $\B[A]$-linear section. By applying the lemma \ref{B_PRIM_SECTION} on sections over primitives to $S=A\Prim\mu$, the section surjects onto the set of elements of $\B(A\Prim\mu,\leq)$ of the form $(A\Prim\mu)_{\leq X}$ with $X\in\Prim\mu$.  Since these generate it as a $\B[A]$-module, it is an isomorphism.
(In particular, $A\Prim\mu=\Prim\mu$ is an $A$-submodule of $\mu$.)\end{proof}

\subsection{Domains}\label{MON_DOMAIN}

It remains to determine whether it is \emph{necessary} that $M$ be projective, that is, that only projective $A$-modules can yield projective $\B[A]$-modules. We have a complete result in the case that $A$ is an $\F_1$-domain.

\begin{defn}[Domain]A monoid pair $(A;A^+)$ is called an $\F_1$-\emph{domain} if $A^+\setminus 0$ is cancellative, $A$ is a localisation of $A^+$, and $A^+$ is integrally closed (i.e.\ saturated) in $A$.\end{defn}

I remind the reader that $A^+$ is always assumed to be sharp. Since $A^+$ is integrally closed in $A$, the multiplicative torsion of $A$ must belong to $A^+$ and therefore be trivial.

\begin{prop}\label{MON_POSET_PROJ}Let $M$ be a partially ordered $A$-module. 
\begin{enumerate}\item If $\B(M;\leq)$ is finitely generated over $\B[A]$, then $M$ is finitely generated over $A$.
\item Suppose $A$ is a domain. If $\B(M;\leq)$ is projective over $\B[A]$, then $M$ is free over $A$.\end{enumerate}\end{prop}
\begin{proof}\emph{i)} If $I\subseteq\B(M,\leq)$ is a generating set, then so too is the set $I^\prime\subset M$ of irredundant primitive factors of elements of $I$. Thus $AI^\prime\subseteq M$ generates $\B(M,\leq)$ as a $\B$-module, and therefore contains $M$. In particular, if $I$ is finite, then $I^\prime$ is finite and $M$ is finitely generated.

\

\emph{ii)} Let $p:\bigoplus_iA_i\twoheadrightarrow M$ be a surjection from a free module, where each $A_i\simeq A$ is free cyclic. By eliminating indices $i$ such that $pA_i$ is strictly contained in another $pA_j$, and identifying indices with the same image, we may assume that $p$ is an \emph{irredundant} generator, i.e. that no factor can be removed without destroying surjectivity of $p$.

By projectivity of $\B(M;\leq)$, there exists a section $\sigma$ to the induced projection
\[\B(\bigoplus_iA_i;A^+)\twoheadrightarrow\B(M;\leq). \]
Using the natural identification $\bigoplus_i\B(A_i)\cong\B(\bigoplus_iA_i)$, we will study $\sigma$ via its components
\[ \sigma_i:\B(M;\leq)\rightarrow\B(\oplus_iA_i;A^+)\stackrel{\mathrm{pr}_i}{\rightarrow}\B A_i. \]

\begin{lemma}\label{MON_DOMAIN_UNIQUE_GENS}If $x\in A_i$, then $x\in\sigma_ipv$.\end{lemma}
\begin{proof}By $A$-linearity, it will be enough to show this for $x$ a generator of $A_i$. Irredundancy implies that $px$ is not in the image of $A_j$ for any $j\neq i$. Thus $\sigma_ipx$ contains a lift $fx$ of $px$, and $f$ fixes $px$.

Applying $\sigma_i$ to the relation $fpx=px$ shows that $f$ acts by an automorphism of $\sigma_ipx$. Since $A^+$ is sharp, this automorphism must permute the finite set of primitive generators. In particular, the orbit of $fx$ is finite. Thus $f$ is torsion, therefore $1$.\end{proof}

%\begin{proof}By minimality of $I$, we must have $p\sigma_iv_i=v_i$. Thus $\sigma_iv_i$ is generated by a finite set of elements of $A$ containing an element $\tilde v_i$ whose action on $M$ fixes $v_i$. By $A$-linearity of $\sigma$, $\tilde v_i$ permutes $\sigma_i v_i$. The product over the orbit of $\tilde v_i$ is idempotent, that is, $\tilde v_i$ is torsion. Since $A^+$ is sharp, $\tilde v_i=1$. The other elements of $\sigma_iv_i$ end up less that $1$ with respect to the partial order on $M$.\end{proof}

\begin{lemma}\label{MON_POSET_LEMMA}$M$ is a direct sum of the cyclic modules $pA_i$.\end{lemma}
\begin{proof}Since $A_i$ and $A_j$ are free cyclic, there are identifications $A_i\simeq A\simeq A_j$. We will show that if $v_i,v_j$, are the images of $1$ in $M$ under $p$ with respect to these two identifications, then any relation of the form
\[ f_iv_i=f_jv_j \]
in $M$ will force $i=j$. 

By lemma \ref{MON_DOMAIN_UNIQUE_GENS}, applying $\sigma_j$ (resp. $\sigma_i$) to the relation shows that there exist $e_i\in\sigma_j v_i$ (resp. $e_j\in\sigma_iv_j$) such that
\[ f_ie_i=f_j, \quad f_i=f_je_j \]
in $A$, and so in particular $e_ie_jf_i=f_i$. By cancellativity of $A\setminus 0$, $e_ie_j=1$. 

Applying $p$ to the relation $\sigma_jv_i\leq\sigma v_i$ gives $e_iv_j\leq v_i$, and vice versa. We deduce \[v_i=e_ie_jv_i\leq e_iv_j\leq v_i\] in $M$, that is, $v_i=e_iv_j$. By irredundancy of $I$, $i=j$.\end{proof}

It remains to show that the cyclic factors of $M$ are free. This follows from a `high-brow' argument based on the fact that $\B$ is strongly monoidal and commutes with equalisers. 

Since $\B(M)$ is by hypothesis projective, the tensor sum $\B(M)\oplus_{\B[A]}-$ is left exact. By taking primitives, this implies that $M\tens_A-$ commutes with equalisers. It follows from the argument for the equational criterion for flatness \ref{ALG_EQN} that any relation
\[ A_i\underset{f_j}{\stackrel{f_i}{\rightrightarrows}} A_i\stackrel{p}{\rightarrow} M \]
in a cyclic summand of $M$ must be trivial.
\end{proof}

%In equations, the condition of lower saturation says the following: 

%When $A$ is a domain and $M$ is free, the localisation $M\rightarrow M[f^{-1}]$ is injective. It follows that more generally, the symbol $[x/f]$, if it exists, is unique. It follows that the condition of lower saturation is equivalent to $M$ being lower in $M\tens_AK_A$ and carrying the induced ordering.

\begin{cor}[Classification]\label{MON_DOMAIN_FIRST}Let $A$ be an $\F_1$-domain. A $\B[A]$-module $\mu$ is finite projective if and only if:
\begin{enumerate}\item it has unique irredundant primitive decompositions;
\item $\Prim\mu$ is a free $A$-submodule of $\mu$, whose induced order is
\item non-degenerate, and
\item can be presented by a finite quiver.\end{enumerate}
Alternatively, the last condition iv) can be replaced by the two conditions
\begin{enumerate}\item[iva)] $M$ is a lower submodule of $M\tens_AK_A$, and
\item[ivb)] $(M,\leq)$ is finitely presented, that is, compact as an object of $\mathbf{PO}\Mod_A$.
\end{enumerate}\end{cor}
\begin{proof}By lemma \ref{MON_PRIM_THM}, a projective module $\mu$ is free on $(\Prim\mu,\leq)$, and by proposition \ref{MON_POSET_PROJ} $\Prim\mu$ is itself finite projective. By proposition \ref{MON_POMOD_THM}, the order on $\Prim\mu$ is lower finite and lower saturated; \emph{loc.\ cit.}\ also handles the converse.

The remaining meat of the theorem is therefore that lower finiteness and lower saturation equate to conditions \emph{iii)} and \emph{iv)}, and that \emph{iv)} is equivalent to \emph{iva)}$+$\emph{ivb)}. More precisely, lower saturation is equivalent to \emph{iva)}, and lower finiteness is equivalent to non-degeneracy together with finite presentation.

\

($\mathrm{ls}\Leftrightarrow\mathrm{iva})$) This follows from unravelling explicitly the definition of lower saturation: for all $x,y\in M$ and $f\in A$ such that $x\leq fy$, there exists a symbol $[x/f]\in M_{\leq y}$ such that $[x/f]f=x$; \[ x\leq fy \quad \Rightarrow \quad \exists x/f \quad x/f\leq y. \] When $M\cong \bigoplus_{i\in I}A_i$ is a direct sum of cyclic modules, it follows that every relation can be deduced from one of the form $\tilde x\leq y$ with $y$ a \emph{generator} of its cyclic factor.

When $f$ is invertible on $M$, the symbol $[x/f]$ automatically exists and is unique: $[x/f]=xf^{-1}$. More generally, when $f:M\rightarrow M$ is injective, then $[x/f]$ is unique if it exists, and the condition only asks that
\[ M_{\leq y}=M\cap M[f^{-1}]_{\leq y}. \]
When $A$ is a domain and $M$ is free, this works for all $f\neq 0$, hence the result.

\

($\mathrm{lf}\Leftrightarrow\mathrm{ndg}+\mathrm{fp}$) After lemmas \ref{MON_DOMAIN_SECOND} and \ref{MON_DOMAIN_NONDEGEN}, it remains to show that lower finiteness implies non-degeneracy. Any degenerate relation $fx\leq x$ in $M$ entails an infinite sequence of relations $f^nx\leq x$. Since $M$ is free, we may localise our study to a cyclic factor and assume $M=A$, and by lower saturation, that $x=1$.

The condition that the set $\{f^n\}_{n\in\N}$ be contained in a finite $A^+$-submodule is precisely that $f$ be \emph{integral} over $A^+$. Since we have assumed that $A^+$ is integrally closed in $A$, non-degeneracy is a necessary condition for lower finiteness.\end{proof}

\begin{remark}[Duality over $\F_1$]\label{MON_DUALITY}There is in general no good theory of duality for modules over $\F_1$-algebras - since $A^{\oplus n}\not\simeq A^n=\Hom(A^{\oplus n},A)$ as soon as $n>1$, free modules are typically not reflexive. However, in light of corollary \ref{MON_DOMAIN_FIRST}, at least for certain po-modules over a domain $A$ we can `borrow' the theory of duality for finite projective $\B[A]$-modules, and define a restricted dual
\[ M^*:= \Prim(\Hom(\B M,\B A)). \]
The result implies that if $M$ is a lower finite and lower saturated partially ordered, finite, free $A$-module, then so is $M^*$, and that this operation is an involution (though still not actually a duality with respect to $\tens_A$) on such modules.
\end{remark}

We end this section with some counterexamples.

\begin{eg}[Projective {$\B[A]$}-module with non-projective $A$-module of primitives]A more involved argument in lemma \ref{MON_POSET_LEMMA} shows that the result holds for $M$ with the divisibility order whenever $A$ is idempotent-free. The same line of thought yields a counterexample with idempotents to part \emph{ii)} of proposition \ref{MON_POSET_PROJ}. Let
\[ A=A^+:= \F_1[f_i,e_i]_{i=1,2} / (e_i^2e_j=e_i, e_if_i=f_j)_{i\neq j}.\] The idempotent $\xi=e_1e_2$ fixes all elements except $0$ and $1$, and multiplication by $f_i$ is injective on its image - so induces an isomorphism between the fixed set of $\xi$ and the principal ideal $(f_1)=(f_2)$.

Now define a module $M$ over $A$ with two generators and the relations
\[ (f_1,0)=(0,f_2),\quad (e_2,0)=(0,\xi),\quad (\xi,0)=(0,e_1). \]
Both of the generating cyclic submodules are free, but $M$ itself is not, so by proposition \ref{ALG_CYCLIC} it cannot be projective. Note also: 
\begin{itemize}\item the fixed set of $\xi$ on $M$ is isomorphic to the fixed set of $\xi$ of $A$ itself; in particular, it is projective cyclic;
\item the base change of $A$ to $\Z$ splits as a product $\Z\times\Z[e_1^{\pm 1}]$, whereupon $M$ becomes the union of a trivial line bundle over the $\G_m$ part and a rank two free module over the point - in particular, it is projective.\end{itemize}

I claim that $\B(M;A)$ is a projective $\B[A]$-module. Indeed, a section to the given map $\B[A]^2\twoheadrightarrow\B(M)$ is given
\[ \sigma(1,0):=(1,0)\vee(0,e_1) \]
and vice versa. I leave it to the reader to check that this map is indeed a section and obeys the relations of $M$.

When working with monoids one somewhat inevitably encounters pathological examples like this one, as these kinds of relations can arise after a simple sequence of relative tensor products between finite free monoids $\simeq \N^m$.\end{eg}

\begin{eg}[A finite, non-reflexive $\Z_\vee$-module]
Note that since finite $\B[A]$-modules no longer have finite underlying sets, they need not have primitive decompositions. 

For example, take $\Z_\vee^2/(e_1+\lambda\vee e_2=e_1)$, where $\lambda\leq 0$. The only primitives in this module are the translates of $e_2$; in particular, it is not generated by primitives. It follows that it cannot be embedded in a free module, and is in particular non-reflexive.\end{eg}

\begin{eg}[Finitely generated, flat, but non-projective $\Z_\vee$-module]Let $\Gamma\subseteq\Gamma^\prime$ be an extension of totally ordered Abelian groups. Then $\Gamma^\prime_\vee$ is a flat $\Gamma_\vee$-module.

\begin{proof}We will prove directly the equational criterion (prop. \ref{ALG_EQN}). Let $v:\Gamma_\vee^n\rightarrow\Gamma^\prime_\vee$ be a $\Gamma$-linear homomorphism, $\Gamma_\vee\rightrightarrows\Gamma^n_\vee$ a relation. These data are determined by a set of (up to) $n$ elements $\{v_i\}_{i=1}^n\subseteq\Gamma^\prime$ and a relation
\[ \max\{v_i+f_i\}=\max\{v_i+g_i\}_{i=1}^n\in \Gamma^\prime \]
with $(f_i),(g_i)\in\Gamma^n$.
Let $i_f$, resp. $i_g$, be the index for which the maximum on the left (resp. right) is realised. If $i_f=i_g$, then $f_{i_f}=g_{i_g}$ and the projection $\Gamma^n\rightarrow\Gamma$ on the $i_f$th factor factors $v$. Otherwise, we may identify the $i_f$th and $i_g$th factor via the primitive of the given relation, and project out the other co-ordinates.
\end{proof}

On the other hand, if the rank of $\Gamma^\prime$ is strictly greater than that of $\Gamma$, then submodules generated by $n>1$ elements are in general infinitely presented, and so cannot be projective. For instance, this applies to the $\Z_\vee$-submodule of $\Z^2_\mathrm{lex}$ generated by $(0,0)$ and $(0,1)$.
\end{eg}

\section{Polyhedra}\label{POLY}
In this section, we will work with additively notated monoids without zero. The associated fractional submodule functor $\B$ factors through adjoining zero. Cf.\ \S\ref{ALG_PTS}.

\subsection{Weight polyhedra}\label{POLY_WEIGHT}
Let $H$ be a totally ordered group - for example, an additive subgroup of $\R$. We will study finite modules over the semifield
\[ H_\vee = \B[H;H^\circ] \cong H\sqcup\{-\infty\} \]
via partially ordered $H$-modules (or `$H$-acts'). This setting enjoys two substantial simplifications over the general case:
\begin{itemize}\item since $H$ is a group, lower saturation is automatic (cf. part \emph{iva} of corollary \ref{MON_DOMAIN_FIRST});
\item since $H$ is totally ordered, non-degeneracy is automatic;\end{itemize}
and so all we have to worry about is finite presentation.

When $M\simeq H^{\sqcup n}$ is free, the dual $H$-module
\[ \V(M):=\Hom_H(M,H)\simeq H^n \]
is an $H$-affine space - that is, a torsor for $\Hom_\Z(\Lambda,H)$ with $\Lambda$ a free $\Z$-module of rank $n$ - whose co-ordinate functions, up to translation, are the elements of $M$. 

When $M$ is free as an $H$-module but carries a possibly non-trivial partial order, it embeds in $\V(M)$ as the subset
\[ \V(M,\leq):=\Hom_{\mathbf{PO}\Mod_H}(M,H) \]
 cut out by relations defining the partial order of $M$. When $(M,\leq)$ is finitely presented, it is a polyhedral subset defined by inequalities between co-ordinates.

%Replacing $H$ with its pointed version, the dual affine space embeds in a natural `partial compactification' $\Hom_H(M,H\sqcup\{-\infty\})$. The stratum at infinity $X=-\infty$ of the dual to a partial order on $M$ is obtained by projecting out $X$; that is, it is the space of unbounded-below rays in $\V(M)$ along which all other co-ordinates are constant.

We formalise the sets arising in this way as follows:

\begin{defns}\label{POLY_WEIGHT_DEF}Let $\Phi\subset V$ be a root system in a Euclidean space with a weight lattice $V_\Z$. An $H$-integral \emph{weight polyhedron} $(\Delta,N,V,\Phi)$ for $\Phi$ is a finite, non-empty intersection of half-spaces orthogonal to the roots, and having co-ordinates in $H$, in an affine space for $\Hom_\Z(V_\Z, H)$.

A \emph{general linear} $H$-integral weight polyhedron is a weight polyhedron with $\Phi$ the root system of the general linear group $\GL_n$ of a finite set $n$.\footnote{That is, the general linear group of a vector space with a direct sum decomposition into lines indexed by $n$. There is then a canonical Cartan subgroup which is naturally identified with $\G_m^n$.} The character lattice $V_\Z$ of $\GL_n$ has a canonical unordered basis $V_\Z\cong\Z^n$. Let us call the elements of this basis the \emph{fundamental weights}. The roots are the differences of pairs of distinct fundamental weights. 

Any $\GL_n$ weight polyhedron $\Delta$ is invariant under the diagonal action of $H$; quotienting by this action one obtains an $\SL_n$ weight polyhedron $\P\Delta$.

An affine map $(\Delta_1,N_1,V_1,\Phi_1)\rightarrow(\Delta_2,N_2,V_2,\Phi_2)$ of general linear weight polyhedra is an $H$-equivariant map $\psi:N_1\rightarrow N_2$ such that $\psi(\Delta_1)\subseteq\Delta_2$ and $d\psi:V_2\rightarrow V_1$ preserves the set of fundamental weights (i.e. is the linear extension of a map $I_2\rightarrow I_1$). An affine map of $\GL_n$ weight polyhedra preserves the centre $\G_m\subseteq\GL_n$ and descends under $\P$ to an affine map of $\SL_n$ weight polyhedra preserving the root system $\Phi\sqcup\{0\}$.

The category of general linear weight polyhedra is denoted $\mathbf{Poly}_H^\GL$. By diagonal invariance of $\GL$ weight polytopes, the mapping sets 
\[ \Aff_H^\GL(\Delta_1,\Delta_2) = \Hom_{\Poly_H^\GL}(\Delta_1,\Delta_2) \] are also $H$-modules. By replacing $\Delta$ with its affine span, a $\GL_n$ weight polyhedron may always be assumed $n$-dimensional.\end{defns}

Note that this definition is more relaxed than the usual notion of weight polytope, since there is no requirement of Weyl group symmetry. When the Weyl group does act, a weight polytope in the sense of definition \ref{POLY_WEIGHT_DEF} is \emph{polar} to the usual weight polytope.

\begin{eg}\label{POLY_WEIGHT_EG}If $\Delta_1$ is one-dimensional (i.e. a $\GL_1$ weight polyhedron), then $\Delta_1\simeq H$ as an $H$-act and the differential of any map $\Delta_1\rightarrow\Delta_2$ is fixed to be the map $\Z^{I_2}\rightarrow\Z$ sending all basis vectors to $1$. To specify such a map, it is therefore enough to say where a single point $0\in\Delta_1$ goes. In particular, $\Aff_H^\GL(H,\Delta_2)=\Delta_2$.

Conversely, a map $\Delta_2\rightarrow\Delta_1$ is necessarily a co-ordinate projection. In particular, the set $\Aff_H^\GL(\Delta_2,H)=:\Aff_H^\GL(\Delta_2)$ of such is a finite $H$-module.\end{eg}

\paragraph{Co-ordinate functions on a weight polyhedron}Let $N$ be a torsor under $\Hom_\Z(V_\Z,H)$. Then in particular $N$ is a weight polyhedron, and
\[ \Aff_H^\GL(N):=\Hom_{\Poly_H^\GL}(N,H)\simeq H^{\sqcup n} \]
is the set of $H$-affine functions whose differentials are fundamental weights. It is a finite free $H$-module indexed by the set of fundamental weights, functorial for $H$-affine maps.

Now let $\Delta\subseteq N(H)$ be a weight polyhedron, and suppose $\Delta$ is not contained in any proper affine subspace of $N$, i.e.\ $\dim\Delta=\dim N$. Then
$\Aff_H^\GL(N)$ acquires a partial ordering based on the values of functions on $\Delta$. Let us call this partially ordered module $\Aff_H^\GL(\Delta)$. It is functorial for maps of weight polyhedra.

This defines a contravariant functor
\[ \Aff_H^\GL:\Poly_H^\GL\rightarrow\mathbf{PO}\Mod^\mathrm{free}_H  \]
into the category of partially ordered free $H$-modules. As an \emph{a posteriori} consequence of theorem \ref{POLY_WEIGHT_THM}, we will see that in fact $\Aff_H^\GL(\Delta)$ is always finitely presented.

\paragraph{Dual to a partially ordered $H$-module}Dualising an object $(M,\leq)$ of $\mathbf{PO}\Mod_H$ with respect to the unit object yields an embedding $\V(M,\leq)\subseteq\Hom_H(M,H)$ into an $H$-affine space; we also saw that if $(M,\leq)$ is finitely presented, $\V(M,\leq)$ is a general linear weight polyhedron. A module homomorphism $M_1\rightarrow M_2$ yields a pullback map $\Hom_H(M_2,H)\rightarrow\Hom_H(M_1,H)$ of the ambient affine spaces whose differential is a permutation.

Thus $\V$ defines a contravariant functor
\[ \mathbb V:\mathbf{PO}\Mod_H^\mathrm{pf/free}\rightarrow \Poly_H^\GL \]
on the category of finite free $H$-modules with finitely presented partial order. (Note that since $H$ is totally ordered, po-$H$-modules are automatically non-degenerate.)

\

Evaluation gives us maps
\[ 1\rightarrow \Aff_H^\GL\V, \quad 1 \rightarrow \V\Aff_H^\GL \]
that put $\V$ and $\Aff_H^\GL$ into adjunction.

\begin{thm}\label{POLY_WEIGHT_THM}The adjoint functors $\V,\Aff_H^\GL$ restrict to an anti-equivalence
\[ \Poly_H^\GL \cong \mathbf{PO}\Mod_H^\mathrm{pf/free} \]
 between the category of $\GL$ weight polyhedra and the category of finitely presented and non-degenerate partially ordered free $H$-modules.\end{thm}
\begin{proof}Let $\Delta\subseteq N$ be a weight polyhedron. A monotone homomorphism $\Aff_H^\GL(\Delta)\rightarrow H$ is a point $p$ of $N$ such that for every pair of affine functions $F,G$ on $V$, $F|_\Delta \leq G|_\Delta$ implies $F(p)\leq G(p)$. Since $\Delta$ is cut out from $N$ by such inequalities, $p\in\Delta$.

The more subtle part of the proof is to show that when $(M,\leq)$ is finitely presented, conversely, the elements of $\V(M)$ `separate points' of $M$. More precisely:

\begin{lemma}[Constructing monotone functionals]\label{POLY_WEIGHT_MAINLEMMA}Let $M\in\mathbf{PO}\Mod^\mathrm{free}_H$ be finitely presented. For any $F\not\leq G\in M$ there exists a monotone homomorphism $\phi_{F,G}:M\rightarrow H_\vee$ such that $\phi_{F,G}G=0$ and $\phi_{F,G} F>0$.\end{lemma}
\begin{proof}By non-degeneracy, a homomorphism $M\rightarrow H$ is monotone if and only if it obeys the relations of the acyclic path set $\sh Q^\emptyset$ of a presenting quiver $\sh Q$. In other words, if $\phi_GG=0$, then we must produce for each cyclic factor $H_i$ an identification $\phi_i:H_i\tilde\rightarrow H$ such that 
\begin{align}\label{arooga} \phi_i\gamma G\leq 0\leq\phi_i\tilde\gamma G \end{align}
for all $\gamma$ in $\sh Q^\emptyset$ from $dG$ to $i$, resp. $\tilde\gamma$ from $i$ to $dG$. Variables whose differentials are not connected to $dG$ in $\sh Q$ may be set arbitrarily.

We distinguish three cases for $F$:

[$dG\rightarrow dF$] Since $H_\vee$ is totally ordered, we can define our identifications by
\[ \phi_i\left(\max_{\gamma:dG\rightarrow i}\gamma G\right):=0. \]
Since, by non-degeneracy, $\max_{\gamma:dG\rightarrow i}\gamma < \tilde\gamma G$ for any $\tilde\gamma:i\rightarrow dG$, the other relations $0<\phi_{G,i}\tilde\gamma G$ are also satisfied.

[$dG\not\rightarrow dF$, $dF\rightarrow dG$] There exists an extension $\tilde\phi_i$ of $\phi_G$ to the factor of $F$. If $\tilde\phi_i(F)\leq 0$, replace it with $\phi_i=\tilde\phi_i-\tilde\phi_i(F)+\epsilon$ where $\epsilon>0$.

[$dF\not\leftrightarrow dG$] We may set $\phi_iF>0$ arbitrarily.\end{proof}

It follows that if $F\leq G$ on $\V(M)\subseteq H^n$, then in fact $F\leq G$ in $M$. In other words, the unit
\[ M\rightarrow \Aff_H^\GL(\V(M)) \]
is an isomorphism of partially ordered modules. In particular, the po-$H$-module associated to a weight polyhedron is finitely presented.\end{proof}

Combining this with the classification theorem \ref{MON_DOMAIN_FIRST}, we obtain a faithful and conservative contravariant embedding
\[ \Poly_H^\GL\hookrightarrow\Mod_{H_\vee} \]
that surjects onto the set of non-zero finite projectives.

\paragraph{Extended weight polyhedron}By duality of finite projective $H_\vee$-modules, the category of weight polyhedra can also be embedded \emph{covariantly} into $\Mod_{H_\vee}$. This inclusion can be realised geometrically by adding strata at infinity directly to a weight polyhedron $\Delta$.

A one-dimensional $H$-affine space $N$ has a natural extension to an invertible $H_\vee$-module $\bar N:=N\sqcup\{-\infty\}$. The same logic allows us to `compactify' a finite product of lines to a free $H_\vee$-module. In terms of functions, it admits the description
\[ \bar N=\bar\V(\Aff_H^\GL(N)):=\Hom_H(\Aff_H^\GL(N),H_\vee); \]
in particular, $\bar N$ is a free $H_\vee$-module.

If now $\Delta\subseteq N$ is a $\GL$ weight polyhedron, then we can define the `closure' $\bar\Delta$ of $\Delta/\bar N$ by extending the inequalities defining $\Delta$ to $\bar N$:
\[ \bar\Delta=\bar\V(\Aff_H^\GL(\Delta),\leq) := \Hom_{\mathbf{PO}\Mod_H}(\Aff_H^\GL(\Delta),H_\vee). \]
If $H=\R$, then this is the same as the topological closure of $\Delta$ in $\bar N$ with respect to the order topology on $\R$. The resulting object is an \emph{extended} general linear weight polyhedron.

%We may say a sequence of elements in $N$ \emph{converges} to a codimension-one stratum at infinity \[ (x_1,\ldots,x_{i-1},\lambda_k,x_{i+1},\ldots,x_n)\rightarrow (x_1,\ldots,x_{i-1},-\infty,x_{i+1},\ldots,x_n) \] if $\lambda_k-\lambda_0\in H$  is unbounded below as $k\rightarrow\infty$. The definition extends in the obvious manner to higher codimension strata.

If $\Delta$ is laterally compact - meaning that the associated $\SL_n$ polyhedron $\P(\Delta)$ is bounded - then the  extension only affixes the single point $-\infty$. More generally, each stratum of $\bar\Delta$ is a $\GL$ weight polyhedron in the stratum of $\bar N$ that contains it.

In general, extended weight polyhedra are dual to finitely presented po-modules over the pointed version $\F_1[H]$ of $H$. For this reason, we adopt the convention that $\{-\infty\}$ is also an extended weight polyhedron.

\paragraph{Mapping sets}The covariant embedding of $\Poly_H^\GL$ in $\Mod_{H_\vee}$ allows us to give a more geometric description of the $H_\vee$-module of homomorphisms between projective modules. 

We define an affine map $\bar\Delta_1\rightarrow\bar\Delta_2$ between extended weight polyhedra to be an affine map - in the sense of the category $\Poly_H^\GL$ - from $\Delta_1$ into a stratum of $\bar\Delta_2$, that is,
\[ \Hom_H^\GL(\bar\Delta_1,\bar\Delta_2):=\coprod_{\Delta_2/\rho}\Hom_H^\GL(\Delta_1,\Delta_2/\rho) \]
where $\rho$ ranges over faces of the recession cone of $N_-\cap\Delta_2$, $N_-$ a negative orthant in $N_2$. Such maps are exactly those arising from homomorphisms of the extended dual $\F_1[H]$-modules $\Hom_H^\GL(\Delta_i,H)\tens_H\F_1[H]$.

The inclusion
\[ \Hom_H\left(\Aff_H^\GL(\Delta_2),\hspace{3pt} \Aff_H^\GL(\Delta_1)\tens_H\F_1[H]\right)\subseteq \Hom_{H_\vee}(\B\Aff_H^\GL(\Delta_2),\hspace{3pt}\B\Aff_H^\GL(\Delta_1))  \]
is generating as a $\B$-module; applying theorem \ref{POLY_WEIGHT_THM} on the left and duality for projective $\B[A]$-modules on the right, the same is true of
\[ \Hom_H^\GL(\bar\Delta_1,\bar\Delta_2)\subseteq \Hom_{H_\vee}(\bar\Delta_1,\bar\Delta_2). \]
It follows that $\Hom_{H_\vee}(\bar\Delta_1,\bar\Delta_2)$ is the set of convex piecewise affine maps $\Delta_1\rightarrow\bar\Delta_2$ whose linear parts preserve the fundamental weights.

\begin{cor}\label{POLY_WEIGHT_COR}The duality of theorem \ref{POLY_WEIGHT_THM} extends $\B$-linearly to a duality 
\[ \Mod_{H_\vee}^\mathrm{f/proj}\cong \Poly_{H,\vee}^\GL \]
between the category of finite projective $H_\vee$-modules and the category of extended $\GL$ weight polyhedra and convex, piecewise-affine maps whose linear parts are fundamental weights.\end{cor}

\subsection{Bundles of weight polyhedra}\label{POLY_FAM}

Let $H\subseteq\R$ be an additive subgroup, and let $\Delta\subseteq N$ be an $H$-rational, strongly convex polyhedron in an $H$-affine space $N$. More precisely, $\Delta\subseteq N(\Q H)$ is a subset defined over the divisible hull $\Q H$ of $H$ cut out by finitely many affine equations over $H$ and having at least one vertex. By multiplying out denominators, the condition of $H$-rationality is the same as rationality over $\Q H$.\footnote{Beware that this category of polytopes differs from the category $\Poly_H^\GL$ discussed in \S\ref{POLY_WEIGHT} in several key ways. First, by $H$-invariance $\GL$ weight polyhedra are never strongly convex. Second, the co-ordinates of the supporting half-spaces for a weight polyhedron are required to lie in $H$, rather than merely in $\Q H$. Finally, while it is always natural to add strata at (minus) infinity to a weight polyhedron, in this section for simplicity we will consider our base polyhedra to be `punctured' at infinity.} 

To $\Delta$ we will associate the monoid pair
\[ \Aff^+(\Delta,H) \subseteq\Aff(\Delta,H) \]
of $H$-affine functions $\Aff(\Delta,H)$ on $\Delta$ with integral slopes and values in $H$, with integers the (saturated, sharp) submonoid $\Aff^+(\Delta,H)$ of functions bounded above by zero. We do not add strata to $\Delta$ at infinity - but see \S\ref{POLY_EXTRA}. Since $\Aff(\Delta,H)=\Aff(N,H)$ is a group, we retain the benifit that we enjoyed in \S\ref{POLY_WEIGHT} that po-modules are automatically lower saturated. However, non-degeneracy is no longer automatic.

We will be interested in the `polyhedral' semiring $\CPA(\Delta,H_\vee)$ of convex piecewise-affine functions on $\Delta$ with integer slopes and values in $H_\vee$. This is generated as a $\B$-module by the pair $\Aff(\Delta,H)$, and is closely related to the free semiring. %The geometric interpretation of our results in this setting can be thought of as a version of duality for $H_\vee$-modules; cf. ex. \ref{MON_DUALITY}.

\begin{defn}Let $\Delta\subset N$ be a strongly convex $H$-rational polyhedron. A \emph{convex family} of weight polyhedra $E$ over $\Delta$ is a convex polyhedron $\Delta_E$ inside an $H$-affine space $N_E/N$ such that for each point $p\in\Delta(H^\prime)$ defined over an ordered extension $H^\prime$ of $H$, $\pi^{-1}(p)$ is an $H^\prime$-rational $\GL$ weight polyhedron. It is enough to check for finite extensions of $H^\prime$. The dimension of the fibre at any interior point is called the \emph{rank} of the family.

A morphism of convex families of weight polyhedra is an $H$-affine map over $N$ that restricts to a morphism in $\Poly_{H^\prime}^\GL$ on the fibre over any $H^\prime$-point.\end{defn}

Note that $\pi$ is necessarily surjective, and the rank does not depend on the point.

\paragraph{Dual to a partially ordered $\Aff$-module}Let us abbreviate $A=(\Aff(\Delta,H),\Aff^+(\Delta,H))$. To a free $A$-module $M$, we associate for every finite extension $H^\prime$ of $H$ a set of $H^\prime$-points
 \[ \V(M)(H^\prime):=\coprod_{p\in\Delta(H^\prime)}\Hom_A(M,H^\prime_p) \] of pairs $(p,\phi)$ consisting of an $H$-algebra map $p:A\rightarrow H^\prime$ and an $A$-module homomorphism $\phi:M\rightarrow H^\prime$, where $H^\prime$ is considered an $A$-module via $p$. For any fixed $H^\prime$, $\V(M)(H^\prime)$ is a trivial $H^\prime$-affine space bundle over $\Delta(H^\prime)$.

If $M$ carries a partial order, we define 
\[ \V(M,\leq)(H^\prime):=\coprod_{p\in\Delta(H^\prime)}  \Hom_{\mathbf{PO}\Mod_A}(M,H_p^\prime)\subseteq \V(M)(H^\prime) \] 
to be the subset of pairs $(p,\phi)$ for which $\phi$ is monotone.

\paragraph{Co-ordinate functions on families of weight polyhedra}Let $\Delta\times H$ denote the trivial affine line bundle over $\Delta$ - more precisely, the affine bundle whose $H^\prime$-points are $\Delta(H^\prime)\times H^\prime$.

If $E$ is a convex family of weight polyhedra, then the set $\Aff_H^\GL(E)$ of affine functions $E\rightarrow\Delta\times H$  whose relative differentials are fundamental weights has a natural structure of a partially ordered, finite, free module over $\Aff(\Delta,H)$. 

\

One obtains a contravariant adjunction
\[  \Aff^\GL_\Delta:\Poly_\Delta^\GL \leftrightarrows \mathbf{PO}\Mod_{\Aff(\Delta,H)}^\mathrm{free}:\V\]
that restricts to an equivalence between the category of affine bundles over $\Delta$ on the left and of free (unordered) $A$-modules on the right. However, the analogue of theorem \ref{POLY_WEIGHT_THM} fails. More precisely, lemma \ref{POLY_WEIGHT_MAINLEMMA} is false in this setting:

\begin{eg}\label{POLY_FAM_EG}Let $\Delta=[a,b]$ be an interval. Then $A^+$ is generated as a monoid by $-1$ and the functions $X-b,a-X$. The free semiring on $A$ is \emph{not} equivalent to the semiring of convex piecewise-affine functions: $0\not\leq X\vee -X$ (provided $a<0<b$). In other words, $\B[A;A^+]$ is \emph{abnormal} (unless $a=b$).

We may thus define a rank two counterexample presented by the quiver
\[\xymatrix{  \bullet \ar@/^/[r]^{X}\ar@/_/[r]_{-X}  & \bullet }\]
which fails to see the relation $(-\infty,0)\leq(0,-\infty)$ which must hold on any map into a totally ordered group.\end{eg}

This can be viewed as a problem on the $\B[A]$-module side and can be solved by considering instead its `normalisation' as in \cite{norm}. In geometric terms, this is achieved by replacing $\mathbf{PO}\Mod_A$ with the associated stack for the rigid analytic topology $\Spec(A;A^+)$.

\paragraph{Normal projectives}Here we carry out a sketch of the programme above, restricting attention to the category $\Mod_{\B[A]}^\nu$ of \emph{normal} $\B[A]$-modules. This is the same as the category of modules over ${}^\nu\B[A]$, the normalisation of $\B[A]$, which in the case $A=\Aff_\Z(\Delta,H)$ is just the semiring $\CPA_\Z(\Delta,H)$ of convex, piecewise-affine functions on $\Delta$. This is the only part of the paper relying on \cite{norm}.

\begin{prop}Let $A$ be an $\F_1$-domain. Every projective ${}^\nu\B[A]$-module is the normalisation of a $\B[A]$-module free on a lower finite, lower saturated partially ordered free $A$-module.\end{prop}
\begin{proof}The proof follows exactly the lines of \S\ref{MON}.

\

\emph{$\Prim\mu$ is an $A$-module:} The primitives of ${}^\nu\B[A]$ are again $A$. It follows from lemma \ref{B_PRIM_SUB} that projective modules over ${}^\nu\B[A]$ also have primitive decompositions. We may therefore repeat the argument of \S\ref{MON_PRIM} to conclude that
\[ \B(\Prim\mu,\leq)\oplus_{\B[A]}{}^\nu\B[A]\rightarrow \mu \]
is an isomorphism for any projective $\mu\in\Mod^\nu_{\B[A]}$. Indeed, by projectivity of $\mu$ there is at least a $\B$-linear section $\sigma$; composing this with a set-theoretic section of $\B[A]\rightarrow{}^\nu\B[A]$ gives a section of $\B(\Prim\mu,\leq)\rightarrow\mu$. By the lemma \ref{B_PRIM_SECTION} on sections over primitives, this necessarily commutes with the two embeddings of $\Prim\mu$. It follows that $\sigma$ is an isomorphism. This part of the proof does not require $A$ to be a domain.

The remainder of the argument depends crucially on the fact that a finitely generated $A^+$-module has a unique minimal set of generators. With a little work, the same holds true in the integrally closed regime as well:

\begin{lemma}[Unicity of generators]\label{POLY_FAM_CRUX}Let $A$ be an $\F_1$-field, $M$ a free $A$-module. Let $K$ be the integral closure of a finitely generated $A^+$-submodule of $M$. There is a unique minimal set of generators of $K$ as an integrally closed fractional submodule.\todo{Probably holds in general}\end{lemma}
\begin{proof}We write in additive notation. By intersecting $K$ with the cyclic factors of $M$ and choosing a trivialisation, we may reduce to the case $M=A$. In this case, integral closure of $K$ coincides with \emph{convexity} \cite[lemma 5.4]{norm}. 

Let $(X_i)_{i\in I}$ and $(Y_j)_{j\in J}$ be two minimal sets of generators of $K$ as a convex $A^+$-submodule. Fixing an index $0\in I$, this give us equations
\begin{align}\nonumber nX_0=\sum_{j\in J}n_jY_j,&\quad n=\sum_{j\in J}n_j \\
\nonumber m_jY_j=\sum_{i\in I}m_{ji}X_i,&\quad m_j=\sum_{i\in I}m_{ji} \end{align}
which we will use to show that $X_0\in\{Y_j\}_{j\in J}$. Indeed, 
\[ \prod_{j\in J_0}m_jnX_0=\sum_{j\in J_0}n_j\prod_{\ell\neq j}m_\ell\sum_{i\in I}m_{j i}X_i = \sum_{i\in I}\left(\sum_{j\in J_0}n_j\prod_{\ell\neq j}m_\ell m_{j i}\right)X_i\]
so either $X_0$ can be eliminated from its generating set by cancelling it from the right, or
\[ \sum_{j\in J}n_jm_{j 0}\prod_{\ell\neq j}m_\ell =1;\qquad \sum_{j\in J}n_jm_{j i}\prod_{\ell\neq j}m_\ell =0,\quad i\neq0. \]
Since all $m_\ell$ are non-zero, there's exactly one index $j\in J$ for which $n_j$ and $m_{j0}$ are both non-vanishing, and $m_{ji}=0$ for all $i\neq 0$. Thus $m_j=m_{j0}=1$ and so $Y_j=X_0$.
\end{proof}

\emph{$(\Prim\mu,\leq)$ is lower finite and lower saturated}. Since $\mu$ is projective, there is a ${}^\nu\B[A]$-linear section to
\[ {}^\nu\B(\Prim\mu;A^+)\rightarrow {}^\nu\B(\Prim\mu,\leq). \]
The proof of the monotone section lemma \ref{B_POSET_SECTION} applies, in light of lemma \ref{POLY_FAM_CRUX}, and shows that the section is automatically a right adjoint.

\

\emph{$\Prim\mu$ is free.} We need clarify only two points:
\begin{itemize}
\item to prove that the submodule spanned by each generator of $M$ was free, we used the fact that $\B$ is strongly monoidal and commutes with equalisers. These properties also hold for ${}^\nu\B = {}^\nu\B[A]\oplus_{\B[A]}(-)$; the first is clear, and the second follows from the fact that the normalisation of a semiring is always flat.\todo{citation}
\item The crucial lemma \ref{MON_DOMAIN_UNIQUE_GENS} that $M$ is a direct sum of cyclics uses the uniqueness of primitive generators for fractional submodules of a free cyclic module, and so passes with another application of lemma \ref{POLY_FAM_CRUX}.\end{itemize}
That said, the proof of proposition \ref{MON_POSET_PROJ} runs verbatim.\end{proof}

\begin{cor}Every projective ${}^\nu\B[A]$-module is the normalisation of a projective $\B[A]$-module.\end{cor}

Not every partially order on a free $A$-module $M$ is induced from its inclusion into the normalisation of $\B(M,\leq)$, as example \ref{POLY_FAM_EG} shows. However, since $\Hom(\alpha,H_\vee)\tilde\rightarrow\Hom({}^\nu\alpha,H_\vee)$ for any semiring $\alpha$ and totally ordered group $H$ \cite[Prop. 1.2.\emph{ii)}]{norm}, the geometric dual $\V$ factors through this replacement.

The missing element of the proof is to show that conversely, $M\rightarrow \Aff_\Delta^\GL(\V(M,\leq))$ is an order isomorphism for $M$ the module of primitives in a normal projective $\B[A]$-module. This follows from a corrected version of lemma \ref{POLY_WEIGHT_MAINLEMMA}:

\begin{lemma}[Constructing monotone functionals]\label{POLY_FAM_MAINLEMMA}Let $M\in\mathbf{PO}\Mod^\mathrm{free}_A$ be finitely presented, and suppose the partial order on $M$ is that induced by its inclusion into $\B(M,\leq)$. For any $F\not\leq G\in M$ there exists a $\phi_{F,G}:M\rightarrow H$ such that $\phi_{F,G}G=0$ and $\phi_{F,G} F >0$.\end{lemma}
\begin{proof}By the same argument as in lemma \ref{POLY_WEIGHT_MAINLEMMA}, it will be enough to find an identification $A_{dF}\tilde\rightarrow A$ and algebra homomorphism $A\rightarrow H$ satisfying the inequalities \ref{arooga} and $\phi_{F,G}>0$. By non-degeneracy, we may assume $dF\neq dG$. 

We have $F\not\leq\bigvee_{\gamma:dG\rightarrow dF}\gamma G$ in ${}^\nu\B(A_{dF})$. Since ${}^\nu\B(A_{dF})$ admits a representation as semiring of convex piecewise-affine functions (\cite{norm} ex. 5.6), there is some point $q\in\Delta(\Q H)$ at which $\bigvee_\gamma\gamma G(q)<F(q)$. If $\gamma_0$ is the path achieving the maximum in this formula, then identifying $A_{dF}$ with $A$ by setting $\gamma_0G=0$ gives us the desired functional.\end{proof}

Finally, following the same line of argumentaion preceding corollary \ref{POLY_WEIGHT_COR}, we extend our definitions to a category of families of \emph{extended} $\GL$ weight polyhedra $\bar\Delta_E\subseteq\bar N_E\twoheadrightarrow\Delta$ and piecewise-affine maps, thereby deriving an analogue of \emph{loc.\ cit.}:

\begin{cor}\label{POLY_FAM_COR}The adjunction between $\Aff_\Delta^\GL$ and $\V$ extends $\B$-linearly to a duality  \[  \Poly_{\Delta,\vee}^\GL \cong \Mod_{\CPA(\Delta)}^{\mathrm{fproj}}\]
between the category of finite projective $\mathrm{CPA}(\Delta)$-modules and the category of convex families of extended $\GL$ weight polyhedra over $\Delta$ with convex piecewise-affine maps whose vertical linear parts are fundamental weights.\end{cor}

\subsection{Extensions}\label{POLY_EXTRA}The results of \S\ref{POLY_FAM} were couched in the setting of $\F_1$-fields, but using corollary \ref{MON_DOMAIN_FIRST} it can easily be extended to the case of domains: we simply need to identify the \emph{lower} free $A$-structures on partially ordered $K_A$-modules.

We obtain a supply of $\F_1$-domains by partially compactifying our polyhedra at infinity. This means that we replace $\Aff(\Delta,H)$ with the submonoid $\Aff(\bar\Delta,H)$ of functions bounded above on $\Delta$. We may also define various partial compactifications of $\Delta$ dual to intermediate saturated submonoids $\Aff(\bar\Delta)\subseteq A\subseteq\Aff(\Delta)$.

A free $A$-structure on a module $M$ over $K_A$ is the same data as a reduction of the structure group of $\V(M)$ from $\Aff(\Delta,H)$ to the group of bounded affine functions. For instance, one may fix a trivialisation $\V(M)\simeq \Delta\times H^n$  (which even determines an $A^+$-structure). The $A$-structure $M_A$ is then the set of functions on $\V(M)$ bounded above by some co-ordinate function of the trivialisation.

Such a reduction defines a \emph{lower} $A$-structure on $M$ if and only if every function bounded above on $\V(M,\leq)$ is bounded above on all of $\V(M)$. This is the case if and only if there exists a constant subset
\[ K\times\Delta\subseteq\V(M,\leq)\subseteq\V(M) \]
whose fibre $K\subseteq H^n$ has non-empty interior. More generally, if $\bar\Delta$ is not compact, such subsets must exist on rays approaching the boundary.

\begin{eg}\label{POLY_EXTRA_EG}Let $\bar\Delta=[-\infty,0]$ with co-ordinate $X$, and let $E$ be the rank two family of polyhedra cut out by the equation $Y_1\leq Y_2+X$. Its restriction to $X=-\infty$ is free of rank one. The dual module $\Aff^\GL_\Delta(E)$ is not lower finite because of the function $Y_1-X$, and so this does not correspond to a projective $\CPA_\Z(\bar\Delta,H)$-module.

However, its base change to $\Delta:=(-\infty,0]$ is projective; it is the module presented by the quiver
\[ \xymatrix{\bullet  & \bullet \ar[l]_{-X} }\] It is trivialised by the change of basis $Y_2\leftrightarrow Y_2+X$; in fact, it is the pullback of a projective $\B$-module. This provides us with a trivial, and in particular projective, extension over $-\infty$.\end{eg}

\begin{remark}[Other geometric classes of modules]If we drop the non-emptiness requirement in the definition of weight polyhedra - and hence that $E\rightarrow\Delta$ be surjective - we obtain a class of modules that can be obtained by `pushing forward' projective modules from sub-polyhedra.
More interestingly, by pushing forward projective modules from boundary strata one can construct modules in polyhedra whose rank jumps up at infinity. The construction of example \ref{POLY_EXTRA_EG} gives also families whose rank decreases at the boundary.

Finally, it will be important in the future to understand modules that correspond to families of `valuated matroids' - but that is another story for another day.\end{remark}

\bibliographystyle{alpha}
\bibliography{proj}

\end{document}